\documentclass[reqno]{amsart}

\usepackage{amsfonts}
\usepackage{amssymb}

\usepackage{amscd}
\usepackage{pictexwd,dcpic}
\usepackage{graphicx}
\usepackage{tikz}
\usepackage{fullpage}
\usepackage{hyperref}
\hypersetup{
    colorlinks=true,
    linkcolor=blue,
    filecolor=magenta,
    urlcolor=cyan,
    citecolor=cyan,}
\usepackage[nameinlink,capitalize]{cleveref}
\usepackage{enumitem}

\title{Blowups of hypersurfaces}

\author{Matthew Weaver}
\address{School of Mathematical and Statistical Sciences, Arizona State University, Wexler Hall, Tempe AZ 85281}
\email{matthew.j.weaver@asu.edu}

\date{}

\newtheorem{thm}{Theorem}[section]
\newtheorem*{thm-nonum}{Theorem}

\newtheorem{prop}[thm]{Proposition}

\newtheorem{cor}[thm]{Corollary}
\numberwithin{equation}{section}

\theoremstyle{definition}
\newtheorem{rem}[thm]{Remark}
\newtheorem{set}[thm]{Setting}
\newtheorem{notat}[thm]{Notation}
\newtheorem{defn}[thm]{Definition}

\newtheorem{ex}[thm]{Example}

\Crefname{thm}{Theorem}{Theorems}
\Crefname{ex}{Example}{Examples}

\def\A{\mathcal{A}}

\def\D{\mathcal{D}}
\def\G{\mathcal{G}}

\def\J{\mathcal{J}}

\def\K{\mathcal{K}}
\def\L{\mathcal{L}}
\def\m{\mathfrak{m}}
\def\n{\mathfrak{n}}

\def\p{\mathfrak{p}}

\def\R{\mathcal{R}}
\def\S{\mathcal{S}}
\def\F{\mathcal{F}}

\def\x{\underline{x}}

\def\dim{\mathop{\rm dim}}
\def\depth{\mathop{\rm depth}}

\def\hgt{\mathop{\rm ht}}
\def\proj{\mathop{\rm Proj}}

\def\reg{\mathop{\rm reg}}
\def\spec{\mathop{\rm Spec}}

\def\chr{\mathop{\rm char}}
\def\bideg{\mathop{\rm bideg}}

\def\rt{\mathop{\rm rt}}

\usepackage[all]{xy}

\newcommand*{\dit}{{\scalebox{0.5}{$\bullet$}}}

\usepackage{nicematrix}

\begin{document}

\begin{abstract}
A classical result of Micali \cite{Micali64} asserts that a Noetherian local ring is regular if and only if the Rees algebra of its maximal ideal is defined by an ideal of linear forms. In this case, this defining ideal may be realized as a determinantal ideal of generic height, and so the Rees ring is easily resolved by the Eagon--Northcott complex, providing a wealth of information. If $R$ is a non-regular local ring, it is interesting to ask how far the Rees ring of its maximal ideal strays from this form, and whether any homological data can be recovered. In this paper, we answer this question for hypersurface rings, and provide a minimal generating set for the defining ideal of the Rees ring. Furthermore, we determine the Cohen--Macaulayness of this algebra, along with several other invariants.
\end{abstract}

\maketitle

%%%%%%%%%%%%%%%%%%%%%%%%%%%%%%%%%%%%%%%%%%%%%%%%%%%%%%%%%%%%%%%%%%%%%%%%%%%%%%%%%%%%%%%%%%%%%%%%%%%%%%%%%%%%%%%%%%%%%%%%%%%%%%%%%%%%%%%%%%%%%%%%%%%%%%%%%%%%%%%%%%%%%%%%%%%%%%%%%%%%%%%%%%%%%%%%%%%%%%%%%%%%%%%%%%%%%%%%%%%%%%%%%%%%%%%%%%%%%%%%%%%%%%%%%%%%%%%%%%%%%%%%%%%%%%%%%%%%%%%%%%%%%%%%%%%%%%%%%%%%%%%%%%%%%%%%%%%%%%%%

\section{Introduction}

For an ideal $I=(f_1,\ldots,f_n)$ of a Noetherian ring $R$, recall that the \textit{Rees algebra} of $I$ is the graded subring $\R(I) = R[f_1t,\ldots,f_nt]$ of $R[t]$, for $t$ an indeterminate. As a graded subalgebra, we note that $\R(I) = R\oplus It\oplus I^2t^2 \oplus \cdots$, hence this ring encodes the data on all powers of $I$. As such, the Rees algebra has proven to be an essential tool for algebraists within the study of multiplicities and reductions. In addition to this algebraic perspective, we note that there is also significant motivation arising within algebraic geometry. Indeed, $\R(I)$ is commonly called the \textit{blowup algebra}, as $\proj(\R(I))$ is precisely the blowup of $\spec(R)$ along the subscheme $V(I)$. Hence the Rees ring is of great importance to geometers within the study of resolutions of singularities and questions related to birationality.

Whereas the above definition of $\R(I)$ is simple to state, this description offers little insight into the structure of this algebra, or its relevant properties. To remedy this issue, one hopes to \textit{implicitize} this algebra, which correlates to writing $\R(I)$ as a quotient of a polynomial ring over $R$ in new indeterminates. More specifically, in the setting above, there is a natural epimorphism $\Psi: R[T_1,\ldots,T_n] \rightarrow \R(I)$ given by $T_i\mapsto f_i t$. This of course induces an isomorphism $\R(I) \cong R[T_1,\ldots,T_n]/\J$ where $\J= \ker \Psi$ is the \textit{defining ideal} of $\R(I)$. With this description, it suffices to produce a generating set of $\J$, the \textit{defining equations} of $\R(I)$, to determine the algebraic structure of the Rees ring, which often reveals many of its properties.

Determining the generators of $\J$ has been called the \textit{implicitization problem} for Rees rings, and has been well studied, however a full set of defining equations of $\R(I)$ is known in very few settings. There has been some success for ideals with low codimension, namely codimension two \cite{BM16,CPW23,MU96,Nguyen14,Nguyen17,Weaver23} and codimension three \cite{KPU17,Morey96, Weaver24}, under various assumptions. The difficulty lies in that $\J$ consists of the polynomial relations amongst the generators $f_1,\ldots,f_n$ of $I$. As such, $\J$ encodes the \textit{syzygies} on every power of $I$. Thus one is often limited to the study of ideals with known resolutions and accessible syzygies. In particular, if the syzygies of $I$ are as simple as possible, namely Koszul, the Rees algebra has a particularly simple description.

If $I=(x_1,\ldots,x_n)$ for $x_1,\ldots,x_n$ an $R$-regular sequence, it is well known that $\R(I) \cong R[T_1,\ldots,T_n]/I_2(\psi)$, for 
\begin{equation}\label{intro psi defn}
\psi=\begin{bmatrix}
  x_1&\cdots & x_n\\
  T_1&\cdots & T_n
\end{bmatrix} 
\end{equation}
and where $I_2(\psi)$ denotes its ideal of $2\times 2$ minors \cite[Chap. I: Thm. 1, Lem. 2]{Micali64}. In particular, the Rees ring of the maximal ideal of any regular local ring is of this form. Interestingly, the converse to this statement holds as well.

\begin{thm-nonum}[{Micali \cite[Chap. I, Thm. 2]{Micali64}}]
Let $R$ be a Noetherian local ring with maximal ideal $\m = (x_1,\ldots,x_n)$. Then $R$ is a regular local ring if and only if $\m$ is of linear type, i.e. $\R(\m) \cong \S(\m)$, where $\S(\m)$ denotes the symmetric algebra of $\m$. In this case, we have $ \R(\m) \cong R[T_1,\ldots,T_n]/I_2(\psi)$.
\end{thm-nonum}

With this, it is curious as to how far the defining ideal of the Rees ring differs from the shape above when $R$ is not a regular local ring. In the present article, we consider a hypersurface ring $R=k[x_1,\ldots,x_n]/(f)$, for $k$ a field and $f$ a homogeneous polynomial, and consider the Rees ring $\R(\m)$, for $\m$ the homogeneous maximal $R$-ideal. Whereas this setting may seem rather mild, we note that much of the previous literature on equations of Rees rings pertains specifically to ideals in a polynomial ring $k[x_1,\ldots,x_n]$, with few exceptions \cite{Weaver23,Weaver24}. As the Rees ring is the algebraic realization of the blowup, there is significant geometric motivation here, as different rings must be taken to account for different blowups. In particular, here the Rees ring $\R(\m)$ correlates to the blowup of the hypersurface $\spec(R) \cong V(f)$ at the point $\{\m\}$.

From the theorem above, the defining ideal $\J$ of $\R(\m)$ will fail to have the previous shape, and the containment $I_2(\psi) \subseteq \J$ is strict. However, we show that the defining ideal is not too distant from this form, and differs by a \textit{downgraded sequence} of polynomials obtained from a recursive algorithm, akin to the methods used in \cite{HS14,RS22}.
The main results of this article, \Cref{main result} and \Cref{Cohen--Macaulayness of R(m)}, are reformulated as follows.

\begin{thm}
Let $S=k[x_1,\ldots,x_n]$ with $n\geq 2$ for $k$ a field. Let
$f\in S$ be a homogeneous polynomial of degree $d\geq 1$, and let $R=S/(f)$ with homogeneous maximal ideal $\m=\overline{(x_1,\ldots,x_n)}$, where $\overline{\,\cdot\,}$ denotes images modulo $(f)$. The Rees algebra of $\m$ is $\R(\m) \cong R[T_1,\ldots,T_n]/\J$ where
$$\J = \overline{I_2(\psi) + (f_1,\ldots,f_d)}$$
where $\psi$ is as in (\ref{intro psi defn}), and $f_1,\ldots,f_d$ is a downgraded sequence of polynomials associated to $f$. Additionally, $\R(\m)$ is almost Cohen--Macaulay and is Cohen--Macaulay if and only if $d\leq n-1$.
\end{thm}

The proof of the statement above relies on the so-called \textit{method of divisors} of Rees rings, as introduced in \cite{KPU11}. We approximate, in a sense, the Rees ring $\R(\m)$ by a simpler algebra and compare these two rings, and their defining ideals. Then an adaptation of the method of \textit{downgraded sequences}, as used in \cite{HS14,RS22}, is employed and it is shown that this recursive algorithm yields a generating set of the defining ideal. With this description of the defining ideal, we describe the Cohen--Macaulayness of $\R(\m)$ and other homological invariants.

We now describe how this paper is organized. In \Cref{Main Section} we begin the study of the Rees algebra of the maximal ideal $\m$ in a hypersurface ring $R$. In particular, we introduce a description of the defining ideal as a particular colon ideal, which will be used throughout the later sections. In \Cref{Approx Section} we map onto the Rees ring $\R(\m)$ by an algebra defined by $I_2(\psi)$, and compare certain properties between these rings. In \Cref{Algorithm Section} we introduce a recursive algorithm beginning from the hypersurface equation $f$ defining $R$. This algorithm produces equations of the defining ideal, and we give a criterion for when this process yields a full generating set. In \Cref{Defining Ideal Section} it is shown that this criterion satisfied and that this algorithm produces a \textit{minimal} generating set of $\J$. Lastly, properties of $\R(\m)$ such as relation type, Cohen--Macaulayness, and regularity are studied.

%%%%%%%%%%%%%%%%%%%%%%%%%%%%%%%%%%%%%%%%%%%%%%%%%%%%%%%%%%%%%%%%%%%%%%%%%%%%%%%%%%%%%%%%%%%%%%%%%%%%%%%%%%%%%%%%%%%%%%%%%%%%%%%%%%%%%%%%%%%%%%%%%%%%%%%%%%%%%%%%%%%%%%%%%%%%%%%%%%%%%%%%%%%%%%%%%%%%%%%%%%%%%%%%%%%%%%%%%%%%%%%%%%%%%%%%%%%%%%%%%%%%%%%%%%%%%%%%%%%%%%%%%%%%%%%%%%%%%%%%%%%%%%%%%%%%%%%%%%%%%%%%%%%%%%%%%%%%%%%%%%%%%%%%%%%%%%%%%%%%%%%%%%%%%%%%%%%%%%%%%%%%%%%%%%%%%%%%%%%

\section{Ideals in hypersurface rings}\label{Main Section}

We now begin our study of the Rees algebra of the maximal ideal of a hypersurface ring. The objective of this section is to introduce the defining ideal and provide a tangible description, which will be used in the proceeding sections. Our setting for the duration of the paper is the following.

\begin{set}\label{main setting}
Let $S=k[x_1,\ldots,x_n]$ for $k$ a field and  $n\geq 2$, and let $\n = (x_1,\ldots,x_n)$. Let
$f\in S$ be a homogeneous polynomial of degree $d\geq 1$, and let $R=S/(f)$ with homogeneous maximal ideal $\m=\overline{\n}$, where $\overline{\,\cdot\,}$ denotes images modulo $(f)$.
\end{set}

We note that the assumption that $n\geq 2$ ensures that $\m$ is an ideal of positive grade, which will be crucial to our arguments. For instance, this ensures that $\dim \R(\m)=n$ \cite[1.2.4]{VasconcelosBook94}. Additionally, we note that for most purposes, one can assume that $d\geq 2$ to avoid reducing to the setting of a polynomial ring. However, we proceed more generally and make no such assumption here.

Before we begin our study of the Rees algebra, we start by producing a free presentation of the maximal ideal $\m$. We note that a free resolution of $\m$, infinite of course, is available from either Shamash \cite{Shamash69} or Tate \cite{Tate57}. However, as we only require a free presentation, we offer a short argument making use of the mapping cone construction \cite[\textsection A3.12]{Eisenbud}.

\begin{prop} \label{Presentation of m}
With respect to the generating set $\{\overline{x_1},\ldots,\overline{x_n}\}$, the ideal $\m$ may be presented by 
$$R^{\binom{n}{2} +1} \overset{\overline{\varphi}}{\longrightarrow} R^{n} \rightarrow \m\rightarrow 0$$
where $\varphi = [\delta_1(\x)\,|\,\partial f]$ is the matrix with $\delta_1(\x)$ the first Koszul differential on $\x= x_1,\ldots,x_{n}$, and $\partial f$ is any column with homogeneous entries, of the same degree, such that $[x_1\ldots x_{n}] \cdot \partial f =f$.
\end{prop}

\begin{proof}
Consider the short exact sequence
$$0\longrightarrow (f) \longrightarrow \n \longrightarrow\m \longrightarrow 0.$$
The $S$-ideal $\n$ is resolved by a truncation of the Koszul complex $\K_\dit(\x)$ on $\x = x_1,\ldots,x_n$, and the ideal $(f)$ is resolved by $F_\dit\,:0\longrightarrow S \longrightarrow (f) \longrightarrow 0$, where the differential is multiplication by $f$. Extending $F_\dit$ by adding zeros, we note that there is a morphism of complexes $u_\dit\,:F_\dit \rightarrow \K_\dit(\x)$ with $u_0 = \partial f$ and $u_i =0$ for $i\geq 1$. With this, the mapping cone construction gives an $S$-resolution of $\m$, and in particular we have a free $S$-presentation
$$S^{\binom{n}{2} +1} \overset{\varphi}{\longrightarrow} S^{n} \rightarrow \m\rightarrow 0$$
with $\varphi$ as above. Tensoring by $R$ then yields the claimed free $R$-presentation.
\end{proof}

\begin{rem}\label{derivatives remark}
We note that there are many choices for $\partial f$ and, as the notation suggests, there is a choice involving differentials. If $\chr k=0$, so that $d=\deg f$ is a unit, one may use the Euler formula 
$$d \cdot f=\sum_{i=1}^{n} \frac{\partial f}{\partial x_i}\cdot x_i $$
and choose $\partial f$ to consist of the partial derivatives of $f$. Regardless, any choice of $\partial f$ will do for our purposes. We will revisit this construction in \Cref{Algorithm Section} and apply it more thoroughly.
\end{rem}

With a free presentation of $\m$, we may now begin the study of the Rees ring $\R(\m)$. The importance of a free presentation is that it allows one to introduce the symmetric algebra $\S(\m)$, noting that there is a natural epimorphism $\S(\m)\rightarrow \R(\m)$. Indeed, as a quotient of $R[T_1,\ldots,T_n]$, $\S(\m)$ is defined by the ideal of entries of the matrix product $[T_1\ldots T_n]\cdot \overline{\varphi}$.

As noted in the introduction, one often aims to write $\R(\m)$ as a quotient of $R[T_1,\ldots,T_n]$ as well. However, it will be more accessible to describe $\R(\m)$ as a quotient of a polynomial ring over $S$ first, similar to the observations made in \cite{Weaver23,Weaver24}. To allow for this, we briefly introduce some notation and conventions, differing slightly from those in the previous section. Consider the sequence of maps 
\begin{equation}\label{Composition}
S[T_1,\ldots,T_n] \longrightarrow R[T_1,\ldots,T_n] \longrightarrow \R(\m)
\end{equation}
where the first map quotients by $(f)$, and the second is the natural map as described in the introduction. As $S$ is the polynomial ring of \Cref{main setting}, it will be more convenient to work over $S[T_1,\ldots,T_n]$. We adopt the bigrading on this ring given by $\deg x_i=(1,0)$ and $\deg T_i=(0,1)$ throughout.

\begin{defn}
Let $\J$ denote the kernel of the composition of the maps in (\ref{Composition}). Moreover, we define the $S[T_1,\ldots,T_n]$-ideal $\L = I_2(\psi) +(f,f_1)$ where $\psi$ is the matrix of indeterminates
\begin{equation}\label{psi defn}
\psi = \begin{bmatrix}
  x_1&\ldots & x_n\\
  T_1&\ldots & T_n
\end{bmatrix}    
\end{equation}
and where $f_1$ is the bihomogeneous polynomial $f_1 = [T_1\ldots T_n]\cdot \partial f$, for the choice of $\partial f$ in \Cref{Presentation of m}.
\end{defn}

We will refer to $\J$ as the \textit{defining ideal} of $\R(\m)$, since $ \R(\m) \cong S[T_1,\ldots,T_n]/\J$. From (\ref{Composition}), notice that the defining ideal in the usual sense is precisely $\overline{\J}$, hence it suffices to produce a generating set of $\J$. Moreover, we note that $\L$ defines $\S(\m)$ as a quotient of $S[T_1,\ldots,T_n]$. Indeed, this follows from \Cref{Presentation of m} and noting that the $2\times 2$ minors of $\psi$ are precisely the entries of the matrix product $[T_1,\ldots,T_n]\cdot \delta_1(\x)$. We make use of this fact throughout in the later sections.

\begin{prop}\label{J a saturation}
With the assumptions of \Cref{main setting}, and $\J$ and $\L$ the $S[T_1,\ldots,T_n]$-ideals above, we have $\J = \L:\n^\infty = (I_2(\psi) + (f)):\n^\infty$.
\end{prop}

\begin{proof}
As mentioned, $\overline{\J}$ and $\overline{\L}$ are the ideals defining $\R(\m)$ and $\S(\m)$ as quotients of $R[T_1,\ldots,T_n]$, respectively. As $\overline{\L}\subseteq \overline{\J}$ and $\m$ is an ideal of positive grade (hence $\m \R(\m)$ is of positive grade), it follows that $\overline{\L}:\m^\infty \subseteq \overline{\J}$. For the reverse containment, notice that $\m$ is of linear type locally for any $\p\in \spec(R)\setminus\{\m\}$. Indeed, we have $\m_\p = R_\p$ for any such prime $\p$, and the unit ideal is obviously of linear type. Hence the quotient $\overline{\J}/\overline{\L}$ is supported only at $\m$, and so some power of $\m$ annihilates it. Thus $\overline{\J} \subseteq \overline{\L}:\m^\infty$, and so $\overline{\J} = \overline{\L}:\m^\infty$. Now since $f$ is contained in all ideals involved, it is clear that $\J = \L:\n^\infty$, which gives the first equality.

For the second equality, we first note that $f_1 \in (I_2(\psi)+(f)): \n$. Indeed, from the definition of $f_1$ and $\partial f$, for any $x_j$ we have
\begin{equation}\label{f1 in colon}
x_jf_1 = [x_jT_1\ldots x_j T_n] \cdot \partial f = [x_1T_j\ldots x_n T_j] \cdot \partial f +[(x_jT_1- x_1T_j)\ldots (x_jT_n- x_n T_j)] \cdot \partial f 
\end{equation}
where the former summand is $T_j f$ and the latter belongs to $I_2(\psi)$. Thus we have $\L \subseteq (I_2(\psi) +(f)):\n$, and so it follows that 
$$\L:\n^s \subseteq (I_2(\psi) +(f)):\n^{s+1} \subseteq \L:\n^{s+1}$$
for any $s\geq 0$. Now taking $s$ large enough, the second equality follows.
\end{proof}

Whereas the saturation $\L:\n^\infty$ may seem more natural, the description $\J = (I_2(\psi)+(f)):\n^\infty$ will prove to be much simpler to use in practice. More specifically, it will be more straightforward to write this saturation as a colon of a specific power of $\n$. This will follow as we may resolve $(I_2(\psi) +(f))$ easily, and then make use of the tools developed in \cite{KPU20} to bound the generation degrees of certain local cohomology modules.

\begin{prop}\label{resolution prop}
Writing $B=S[T_1,\ldots,T_n]$, the ring $B/(I_2(\psi) +(f))$ has a minimal bigraded free $B$-resolution
$$0\rightarrow F_n \rightarrow \cdots\rightarrow F_1\rightarrow F_0$$
where 
$$F_n = B^{n-1} (-d-n+1,-n+1), \quad  F_i = \begin{array}{c}
B^{\binom{n}{i}(i-1)}(-d-i+1,-i+1)\\
\oplus \\
B^{\binom{n}{i+1}i}(-i,-i)
\end{array},  \quad  F_1 = \begin{array}{c}
B(-d,0)\\
\oplus \\
B^{\binom{n}{2}}(-1,-1)
\end{array}, \quad F_0 =B $$
for $2\leq i\leq n-1$.
\end{prop}

\begin{proof}
As $\psi$ is a matrix of indeterminates, it is well known that $I_2(\psi)$ is a prime $B$-ideal \cite[2.10]{BV88}. Additionally, we note that $f\notin I_2(\psi)$, which follows from bidegree considerations, hence there is a bigraded short exact sequence
$$0\rightarrow B/I_2(\psi) (-d,0) \overset{\cdot f}{\longrightarrow} B/I_2(\psi) \rightarrow B/(I_2(\psi)+(f))\rightarrow 0.$$
Moreover, we have that $B/I_2(\psi)$ is resolved by the bigraded Eagon--Northcott complex \cite[A2.10]{Eisenbud}, with the appropriate bidegree shifts for the first copy above. Moreover, we note that multiplication by $f$ lifts to a morphism of the Eagon--Northcott complex to itself, hence the mapping cone construction gives the claimed resolution.

Since $I_2(\psi)$ is prime of height $n-1$, and $f\notin I_2(\psi)$, we have $\hgt (I_2(\psi) +(f)) =n$. As this agrees with the length of the resolution above, it follows that this resolution is of minimal length. Moreover, from the bidegree shifts involved, it follows that this is a minimal bigraded free resolution.
\end{proof}

We now describe $\J= (I_2(\psi) + (f)):\n^\infty$ as a colon ideal of a specific power of $\n$, by applying techniques developed in \cite{KPU20} to bound the degrees of certain local cohomology modules. Whereas these tools can be slightly technical, the procedure is relatively straightforward if a suitable graded free resolution is available. We must first briefly recall some notation and terminology used in \cite{KPU20}. For $A$ a nonnegatively graded ring with $A_0$ local, and $M$ an $A$-module, we write $b_0(M)=\inf\{ j \,|\, A(\bigoplus_{i\leq j} M_i) =M\}$ to denote the \textit{maximal generator degree} of $M$. Letting $\mathfrak{M}$ denote the homogeneous maximal $A$-ideal, write $a(M) = \sup\{i\,|\, [H_{\mathfrak{M}}^{\dim A}(M)]_i \neq 0\}$ to denote the \textit{a-invariant} of $M$.

\begin{prop}\label{J correct colon prop}
With the assumptions of \Cref{main setting}, we have $\J  = (I_2(\psi) + (f)):\n^d$. 
\end{prop}

\begin{proof}
Writing $B=S[T_1,\ldots,T_n]$ once more, by \Cref{J a saturation}, modulo $(I_2(\psi) + (f))$ the kernel of the induced map $B/(I_2(\psi) + (f)) \rightarrow \R(\m)$ is precisely the local cohomology module $\A=H_\n^0\big(B/(I_2(\psi) + (f))\big)$. The $B$-module $\A$ is naturally bigraded, where $\deg x_i =(1,0)$ and $\deg T_i =(0,1)$ again. Hence for a fixed $q$,
we have
$$\A_{(*,q)} = \bigoplus_p \A_{(p,q)} \cong H_\n^0\big([B/(I_2(\psi) + (f))]_{(*,q)}\big).$$
As $\A$ lives in finitely many degrees, we claim that $\A$ vanishes past degree $d-1$ in the first component of the bigrading, i.e. that $\A_{(p,q)} = 0$ for any $q$ and $p> d-1$.

With the $B$-resolution of $B/(I_2(\psi) +(f))$ given in \Cref{resolution prop}, taking the graded strand in degree $(*,q)$ yields a graded $S$-resolution of $[B/(I_2(\psi) + (f))]_{(*,q)}$. With this and \cite[3.8]{KPU20}, it follows that $\A_{(p,q)} =0$ for all $p>b_0([F_n]_{(*,q)}) + a(S)$, where $F_n$ is the free $B$-module in \Cref{resolution prop}, and $[F_n]_{(*,q)}$ is the free $S$-module $[F_n]_{(*,q)} = \bigoplus_p [F_n]_{(p,q)}$. However, from the degree shifts of the given resolution it follows that $b_0([F_n]_{(*,q)}) \leq d+n-1$, for any $q$. Moreover, since $S$ is a standard graded polynomial ring in $n$ variables, it is well known that $a(S) =-n$ \cite[3.6.15]{BH93}. Hence $\A_{(p,q)} =0$ for all $p>d-1$ and any $q$, and so it follows that $\n^d \A=0$. Thus $\n^d \J \subseteq I_2(\psi)+(f)$ and so $\J  \subseteq (I_2(\psi) + (f)):\n^d$. The reverse containment follows from \Cref{J a saturation}.
\end{proof}

\section{Approximations of Rees rings}\label{Approx Section}

Now that the description $\J  = (I_2(\psi) + (f)):\n^d$ has been introduced, we may begin the search for generators. We follow the so-called \textit{method of divisors} of Rees rings introduced in \cite{KPU11}, and used with great success in \cite{BM16,CPW23, CPW24,Weaver23,Weaver24}. In short, we introduce a simpler ring mapping onto $\R(\m)$ with kernel an ideal of height one. Then we aim to produce a generating set of this kernel, and lift it to a generating set of the defining ideal $\J$.

Notice that there is a natural map $\R(\n) \rightarrow \R(\m)$, which can be seen as $I_2(\psi)$ defines $\R(\n)$ \cite[Chap. I, Theorem 2]{Micali64} and $I_2(\psi)\subseteq \J$. Alternatively, one takes the natural map $S\rightarrow R=S/(f)$, and compares the powers of $\n$ and $\m$. With this, we have the following commutative diagram
\begin{equation}\label{R(n) and R(m) diagram}
    \SelectTips{cm}{}
    \xymatrix{
    S[T_1,\ldots,T_n]\ar_{\mod I_2(\psi)}[d] \ar^{\mod(f)}[rr] & & R[T_1,\ldots,T_n]\ar[d]\\
    \R(\n) \ar[rr] & & \R(\m) }
\end{equation}
involving the maps in (\ref{Composition}), and the observation above. Whereas the previous section concerned the composition of the upper rightmost maps in the diagram above, the current section will consider its factorization along the lower leftmost maps. We begin by introducing notation to permit this.

\begin{defn}
With the diagram (\ref{R(n) and R(m) diagram}), write $\widetilde{\,\cdot\,}$ to denote images modulo $I_2(\psi)$ in $\R(\n)$. Additionally, we define the $S[T_1,\ldots,T_n]$-ideal $\K = I_2(\psi)+(x_n,T_n)$.
\end{defn}

The $\R(\n)$-ideal $\widetilde{\K}$ will inevitably pave the path to a generating set of the defining ideal. For the duration of this section, we explore the various properties of this ideal and its relation to $\widetilde{\n}$ and $\widetilde{\J}$, the latter being the kernel of the induced map $\R(\n)\rightarrow \R(\m)$.

\begin{prop}\label{K CM ht 1 prop}
The Rees ring $\R(\n)$ is a Cohen--Macaulay domain of dimension $n+1$. The $\R(\n)$-ideals $\widetilde{\n}$ and $\widetilde{\K}$ are Cohen--Macaulay of height one. Moreover, $\widetilde{\n}$ is a prime ideal.
\end{prop}

\begin{proof}
The ring $\R(\n)$ is easily seen to be a domain of dimension $n+1$, as $S$ is a domain of dimension $n$. Moreover, the Cohen--Macaulayness follows as $\R(\n)$ is defined by an ideal of minors of generic height \cite[A2.13]{Eisenbud}. For the statements regarding $\widetilde{\n}$, notice this ideal defines the special fiber ring $\F(\n) =\R(\n)/\n \R(\n)$. As $\n$ is an ideal of linear type, it follows that $\F(\n) \cong k[T_1,\ldots,T_n]$, hence $\widetilde{\n}$ is a prime Cohen--Macaulay ideal of height one. For the statement on $\widetilde{\K}$, notice that 
\begin{equation}\label{K and matrix}
\K = I_2(\psi) +(x_n,T_n) =I_2\left(\begin{bmatrix}
   x_1&\cdots &x_{n-1}\\
    T_1&\cdots& T_{n-1}
\end{bmatrix}\right) +(x_n,T_n)
\end{equation}
which is easily seen to be Cohen--Macaulay of height $(n-2) +2 =n$. Thus modulo $I_2(\psi)$, we have that $\widetilde{\K}$ is Cohen--Macaulay of height one.
\end{proof}

Alternatively, the Cohen--Macaulayness of $\widetilde{\K}$ can be seen from \cite[A2.13, A2.14]{Eisenbud}, but we opt for the simpler proof above. However, we will use these tools later in \Cref{Defining Ideal Section} to describe the Cohen--Macaulayness of particular powers of $\widetilde{\K}$. With the properties of $\widetilde{\K}$ and $\widetilde{\n}$ established in \Cref{K CM ht 1 prop}, we now show that the symbolic powers of these ideals are \textit{linked} \cite{Huneke82}.

\begin{prop}\label{linkage prop}
With the $\R(\n)$-ideals $\widetilde{\K}$ and $\widetilde{\n}$ above, we have the following.
\begin{enumerate}
    \item[(a)] $\widetilde{\n}^i = \widetilde{\n}^{(i)}$, 

    \item[(b)] $(\widetilde{x_n}^i):\widetilde{\K}^{(i)} = \widetilde{\n}^{(i)}$,

    \item[(c)] $(\widetilde{x_n}^i):\widetilde{\n}^{(i)} = \widetilde{\K}^{(i)}$,

\end{enumerate}
for all $i\in \mathbb{N}$. 
\end{prop}

\begin{proof} We proceed as in the proof of \cite[3.9]{BM16} throughout.
\begin{enumerate}
    \item[(a)] Temporarily regrading $S[T_1,\ldots,T_n]$ with $\deg x_i= 1$ and $\deg T_i=0$, we see that $\G(\widetilde{\n}) \cong \R(\n)$, where $\G(\widetilde{\n})$ is the associated graded ring of the ideal $\widetilde{\n}$. Since $\R(\n)$ is a domain, the first claim follows.  

\vspace{1mm}

\item[(b)] In $\R(\n)$ we have $\widetilde{x_i}\widetilde{T_n} = \widetilde{x_n}\widetilde{T_i}$ for $1\leq i\leq n$. Thus it follows that $\widetilde{\K} \widetilde{\n} \subseteq (\widetilde{x_n})$, and so $\widetilde{\K}^i\widetilde{\n}^i \subseteq(\widetilde{x_n}^i)$ for any $i\geq 1$. Now localizing at prime ideals of height one and contracting, we see that $\widetilde{\K}^{(i)}\widetilde{\n}^{(i)} \subseteq (\widetilde{x_n}^i)$. With this we have $\widetilde{\n}^{(i)} \subseteq (\widetilde{x_n}^i):\widetilde{\K}^{(i)}$, hence we need only show the reverse containment. Since $\widetilde{T_n}\notin \widetilde{\n}$, we note that $\widetilde{\K} \nsubseteq \widetilde{\n}$. As $\widetilde{\n}$ is the unique associated prime of $\widetilde{\n}^{(i)}$, it follows that $\widetilde{\K}^{(i)}$ and $\widetilde{\n}^{(i)}$ have no associated primes in common. With this, and noting that $\widetilde{x_n}^i\in \widetilde{\n}^{(i)}$, it follows that $(\widetilde{x_n}^i):\widetilde{\K}^{(i)}\subseteq \widetilde{\n}^{(i)}$ as well.

\vspace{1mm}

\item[(c)] As noted in the proof of (b), there is a containment $\widetilde{\K}^{(i)}\widetilde{\n}^{(i)} \subseteq (\widetilde{x_n}^i)$, and so we have $\widetilde{\K}^{(i)} \subseteq (\widetilde{x_n}^i):\widetilde{\n}^{(i)}$. Again noting that $\widetilde{\K}^{(i)}$ and $\widetilde{\n}^{(i)}$ have no associated primes in common, and that $\widetilde{x_n}^i\in \widetilde{\K}^{(i)}$, it then follows that $(\widetilde{x_n}^i):\widetilde{\n}^{(i)} \subseteq \widetilde{\K}^{(i)}$.\qedhere
\end{enumerate}
\end{proof}

Items (b) and (c) of \Cref{linkage prop} show that the ideals $\widetilde{\n}^{(i)}$ and $\widetilde{\K}^{(i)}$ are \textit{linked} through the complete intersection $(\widetilde{x_n}^i)$. Moreover, since they have no common associated prime, they are said to be \textit{geometrically} linked.

\begin{cor}\label{K SCM and generically a CI}
The ideal $\widetilde{\K}$ is generically a complete intersection. Moreover, $\widetilde{\K}$ is a strongly Cohen--Macaulay $\R(\n)$-ideal. 
\end{cor}

\begin{proof}
The first statement is a consequence of geometric linkage. Since $\widetilde{\K}$ and $\widetilde{\n}$ share no associated prime, by \Cref{linkage prop} we have that $\widetilde{\K}_\p = (\widetilde{x_n})_\p: \R(\n)_\p = (\widetilde{x_n})_\p$ for any associated prime $\p$ of $\widetilde{\K}$. As such, $\widetilde{\K}$ is generically a complete intersection.

For the second statement, recall that $\widetilde{\K}$ is Cohen--Macaulay of height one by \Cref{K CM ht 1 prop}. Moreover, since $\widetilde{\K} = (\widetilde{x_n},\widetilde{T_n})$, it is an almost complete intersection ideal. Since $\widetilde{\K}$ is also generically a complete intersection, it follows from \cite[2.2]{Huneke83} that $\widetilde{\K}$ is strongly Cohen--Macaulay.
\end{proof}

We now give a description of the ideal $\widetilde{\J}$, noting that this is the kernel of the induced map $\R(\n) \rightarrow \R(\m)$ in (\ref{R(n) and R(m) diagram}). Consider the divisorial ideal $\D = \frac{\widetilde{f}\widetilde{\K}^{(d)} }{\widetilde{x_n}^d}$, and note that this is actually an $\R(\n)$-ideal by \Cref{linkage prop}, since $f \in \n^d$.

\begin{thm}\label{J=D}
   With the assumptions of \Cref{main setting}, we have $\widetilde{\J} = \D.$
\end{thm}

\begin{proof}
We first show that $\D \subseteq \widetilde{\J}$. Recall that $\J = (I_2(\psi) +(f)):\n^d$ by \Cref{J correct colon prop}, hence we need only show that $\D \cdot \widetilde{a} \subseteq  (\widetilde{f})$ for any $a\in \n^d$. However, this is clear since
$$\D \cdot \widetilde{a} = \frac{\widetilde{f}\widetilde{\K}^{(d)} }{\widetilde{x_n}^d} \cdot \widetilde{a} = \frac{\widetilde{a}\widetilde{\K}^{(d)} }{\widetilde{x_n}^d} \cdot \widetilde{f} \subseteq (\widetilde{f}),$$
noting that $\tfrac{\widetilde{a}\widetilde{\K}^{(d)} }{\widetilde{x_n}^d}$ is an $\R(\n)$-ideal by \Cref{linkage prop}.

With the containment $\D \subseteq \widetilde{\J}$, it is enough to show equality locally at the associated primes of $\D$. Notice that $\D \cong \widetilde{\K}^{(d)}$, and the latter is unmixed of height one by \Cref{linkage prop} and \cite[0.1]{Huneke82}. As such, $\D$ is also an unmixed $\R(\n)$-ideal of height one, and so it is enough to show that $\D_\p = \widetilde{\J}_\p$ for any prime $\R(\n)$-ideal $\p$ with $\hgt \p =1$. 
Recall from \Cref{K CM ht 1 prop} that $\widetilde{\n}$ is such a prime ideal of $\R(\n)$. If $\p \neq \widetilde{\n}$, we note that $\widetilde{\n}^d_\p = \R(\n)_\p$. Thus by \Cref{linkage prop}, we have $\widetilde{\K}^{(d)}_\p=(\widetilde{x_n}^d)_\p$, and so  $\D_\p = (\widetilde{f})_\p$. Similarly, by \Cref{J correct colon prop} we see that $\widetilde{\J}_\p =(\widetilde{f})_\p$ as well. 

Thus the proof will be complete once it has been shown that $\D$ and $\widetilde{\J}$ agree locally at $\widetilde{\n}$. We first claim that $\widetilde{\J} \nsubseteq \widetilde{\n}$, and so $\widetilde{\J}_{\widetilde{\n}} = \R(\n)_{\widetilde{\n}}$. Indeed, from the isomorphism $\R(\m) \cong \R(\n)/\widetilde{\J}$, it follows that the fiber ring of $\m$ is $\F(\m)\cong \R(\n)/(\widetilde{\J}+\widetilde{\n})$, which has dimension at most $\dim R =n-1$ \cite[5.1.4]{VasconcelosBook94}. However, recall that $\R(\n)/\widetilde{\n} \cong k[T_1,\ldots,T_n]$, and so it follows that $\widetilde{\J} \nsubseteq \widetilde{\n}$.

Thus we must show that $\D_{\widetilde{\n}}$ is the unit ideal as well. Recall from the proof of \Cref{linkage prop} that $\widetilde{\K} \nsubseteq \widetilde{\n}$, hence $\widetilde{\K}^{(d)}_{\widetilde{\n}} = \R(\n)_{\widetilde{\n}}$. With this, by \Cref{linkage prop} we have $(\widetilde{x_n}^d)_{\widetilde{\n}}:\widetilde{\n}^d_{\widetilde{\n}} = \R(\n)_{\widetilde{\n}}$. Thus $\widetilde{\n}^d_{\widetilde{\n}} \subseteq (\widetilde{x_n}^d)_{\widetilde{\n}}$ and so $\widetilde{\n}^d_{\widetilde{\n}} =(\widetilde{x_n}^d)_{\widetilde{\n}}$, as the reverse containment is clear. Similarly, by \Cref{J correct colon prop} and recalling that $\widetilde{\J}_{\widetilde{\n}} = \R(\n)_{\widetilde{\n}}$, we have $\widetilde{\n}^d_{\widetilde{\n}} \subseteq (\widetilde{f})_{\widetilde{\n}}$ and so $\widetilde{\n}^d_{\widetilde{\n}} =(\widetilde{f})_{\widetilde{\n}}$ as $f\in \n^d$. Now since $(\widetilde{f})_{\widetilde{\n}} = (\widetilde{x_n}^d)_{\widetilde{\n}}$, it follows that $\D_{\widetilde{\n}} =\widetilde{\K}^{(d)}_{\widetilde{\n}} = \R(\n)_{\widetilde{\n}}$.
\end{proof}

%%%%%%%%%%%%%%%%%%%%%%%%%%%%%%%%%%%%%%%%%%%%%%%%%%%%%%%%%%%%%%%%%%%%%%%%%%%%%%%%%%%%%%%%%%%%%%%%%%%%%%%%%%%%%%%%%%%%%%%%%%%%%%%%%%%%%%%%%%%%%%%%%%%%%%%%%%%%%%%%%%%%%%%%%%%%%%%%%%%%%%%%%%%%%%%%%%%%%%%%%%%%%%%%%%%%%%%%%%%%%%%%%%%%%%%%%%%%%%%%%%%%%%%%%%%%%%%%%%%%%%%%%%%%%%%%%%%%%%%%%%%%%%%%%%%%%%%%%%%%%%%%%%%%%%%%%%%%%%%%%%%%%%%%%%%%%%%%%%%%%%%%%%%%%%%%%%%%%%%%%%%%%%%%%%%%%%%%%%%

\section{Downgraded sequences}\label{Algorithm Section}

In this section we introduce a recursive algorithm that yields equations of the defining ideal $\J$. This process is an adaptation of the methods used in \cite{HS14,RS22}, beginning with an initial polynomial and producing a new \textit{downgraded} polynomial at each step. We also note that similar recursive processes have been previously employed to produce equations of Rees rings in \cite{BM16,CHW08,Weaver23,Weaver24}. Before we introduce this algorithm, we recall a bit of earlier notation, which will be used throughout.

\begin{notat}
For any bihomogeneous polynomial $g\in S[T_1,\ldots,T_n]$, write $\partial g$ to denote a column of bihomogeneous elements, of the same bidegree, such that 
$$[x_1\ldots x_n] \cdot \partial g = g.$$
As a convention, we set $\partial g =0$ if $g=0$.
\end{notat}

This of course agrees with the convention used in \Cref{Presentation of m}, now extended to $S[T_1,\ldots,T_n]$. Similar to the observation made in \Cref{derivatives remark}, the column $\partial g$ is not unique, and there are often many choices. %In particular, there is a natural choice for $\partial g$ consisting of the partial derivatives, with respect to $x_1,\ldots,x_n$, in the case that $\chr k=0$. 
Despite this non-uniqueness, we will soon see that the ideals obtained from the following process are well defined.

\begin{defn}\label{algorithm defn}
With the notation above and the assumptions of \Cref{main setting}, we produce a sequence of polynomials and ideals recursively, as follows. Set $f_0 =f$ and $\J_0 = I_2(\psi)+(f_0)$. For $i\leq d=\deg f$ suppose that $(\J_0,f_0), \ldots, (\J_{i-1},f_{i-1})$ have been constructed. To construct the $i$th pair, set $f_i = [T_1\ldots T_n]\cdot \partial f_{i-1}$ and let $\J_i = \J_{i-1} +(f_i)$. We call $f_0,\ldots,f_i$ the $i$th \textit{downgraded sequence} of $f$.
\end{defn}

Notice that at each step, the $i$th ideal is $\J_i = I_2(\psi) +(f_0,\ldots,f_i)$. From the construction of $\partial f_i$, it follows that the degree, with respect to $x_1,\ldots,x_n$, decreases by one at each step, hence the terminology of \textit{downgraded} sequence. We formalize this as follows.

\begin{prop}\label{fi bideg remark}
A downgraded sequence of $f$ consists of nonzero polynomials. That is, for $0\leq i\leq d$ we have $f_i\neq 0$. Moreover, each $f_i$ is bihomogeneous of bidegree $(d-i,i)$.
\end{prop}

\begin{proof}
As the construction above is recursive, we prove the first statement on the nonvanishing by induction. The second statement then follows immediately from the first by construction. Recall that $f$ is a homogeneous polynomial of degree $d\geq 1$ in $S$, hence $f=f_0 \neq 0$. Now suppose that $i\geq 1$, and that $f_i=0$ for some $i$, and we may take $i$ to be minimal, in the sense that $f_j\neq 0$ for $j<i$. Now since $f_i=[T_1\ldots T_n]\cdot \partial f_{i-1}$ and $f_i=0$, it follows that $\partial f_{i-1}$ is a syzygy on $T_1,\ldots,T_n$. As this is a regular sequence, $\partial f_{i-1}$ is then contained in the span of its Koszul syzygies. Since $f_{i-1}=[x_1\ldots x_n]\cdot \partial f_{i-1}$, it then follows that $f_{i-1} \in I_2(\psi)$.

If $i=1$, this is a contradiction %by degree considerations 
as $f_0=f$ is nonzero of bidegree $(d,0)$ and $I_2(\psi)$ is generated in bidegree $(1,1)$, hence we must have $i\geq 2$ here. Thus $f_{i-1} = [T_1\ldots T_n] \cdot \partial f_{i-2}$ and since $f_{i-1}\in I_2(\psi)$, it follows that $\partial f_{i-2}$ is contained in the span of Koszul syzygies on $x_1,\ldots, x_n$. However, then $f_{i-2} = [x_1\ldots x_n]\cdot \partial f_{i-2} =0$ which is a contradiction as $i$ was taken to be smallest.
\end{proof}

As noted, there are many choices for $\partial f_i$ within each step of the process described above. Nevertheless, the ideals constructed from downgraded sequences are unique.

\begin{prop}\label{Ji well defined}
The ideals $\J_0,\ldots,\J_d$ are well defined, and depend only on $f$.
\end{prop}

\begin{proof}
We proceed inductively, noting that the ideal $\J_0$ is certainly well defined. Now suppose that $1\leq i \leq d$ and the ideals $\J_0,\ldots,\J_{i-1}$ are well defined. Letting $\partial f_{i-1}$ and $\partial f_{i-1}'$ denote two columns with 
$$[x_1\ldots x_n]\cdot \partial f_{i-1} = f_{i-1} = [x_1\ldots x_n]\cdot \partial f_{i-1}'$$
we must show that $\J_{i-1}+(f_i) = \J_{i-1}+(f_i')$, where $f_i$ and $f_i'$ are obtained as in \Cref{algorithm defn}.

Notice from the expression above that $\partial f_{i-1} - \partial f_{i-1}'$ is a syzygy on $x_1,\ldots,x_n$. As this is a regular sequence, $\partial f_{i-1}$ and $\partial f_{i-1}'$ must differ by an element in the span of Koszul syzygies on $x_1,\ldots, x_n$. Hence $f_i=[T_1\ldots T_n] \cdot \partial f_{i-1}$ and $f_i'=[T_1\ldots T_n] \cdot  \partial f_{i-1}'$ differ by an element of $I_2(\psi)$. As $I_2(\psi) \subseteq \J_{i-1}$, it is clear that $\J_{i-1}+(f_i) = \J_{i-1}+(f_i')$, and the claim follows.
\end{proof}

\begin{rem}
We note that employing the process of \Cref{algorithm defn} is relatively simple in practice. Since \Cref{Ji well defined} shows the ideal does not depend on the choice of $\partial f_i$, the method of downgraded sequences may be taken as an \textit{exchange} process. Namely, for every monomial in the support of $f=f_0$, one need only exchange exactly one of the $x_i$ for the corresponding $T_i$ in order to produce $f_1$. One then moves onto $f_1$ and repeats the process until there is no $x_i$ to exchange, i.e. once  one reaches $f_d$. 
\end{rem}

For the convenience of the reader, we offer a handful of examples to illustrate the simplicity of the algorithm above, and how there are multiple choices at each step, yet the resulting ideal is the same.

\begin{ex} \,
\begin{enumerate}
    \item[(a)] The method of downgraded sequences is simple when $f$ is a monomial, following the remark above. In particular, if $f$ is a pure power, say $f=x_n^d$, notice that $f_1=x_n^{d-1}T_n, \,f_2=x_n^{d-2}T_n^2,\ldots ,f_d = T_n^d$. The ideal of the full downgraded sequence is $\J_d = I_2(\psi)+ (x_n,T_n)^d$ in this case.

\vspace{1mm}

    \item[(b)] As a more interesting example, let $S=k[x_1,x_2,x_3]$ and $f= x_1^2x_2 + x_1x_3^2$. The downgraded sequence of $f$ may be taken as
    \[
    \begin{array}{cccc}
    f_0= x_1^2x_2 + x_1x_3^2, &  f_1=x_1x_2T_1 + x_3^2T_1, &  f_2=x_2T_1^2 + x_3T_1T_3, & f_3 =T_1^2T_2 + T_1T_3^2. 
\end{array}
\]
Alternatively, one may make different choices and take the downgraded sequence to be
\[
\begin{array}{cccc}
    f_0'= x_1^2x_2 + x_1x_3^2, &  f_1'=x_1^2T_2+x_1x_3T_3, &  f_2'=x_1T_1T_2+x_3T_1T_3, & f_3' =T_1^2T_2 + T_1T_3^2. 
\end{array}
\]
Despite there being multiple choices, it follows from \Cref{Ji well defined} that $I_2(\psi) + (f_0,f_1,f_2,f_3) = I_2(\psi) + (f_0',f_1',f_2',f_3')$ where 
$$\psi = \begin{bmatrix}
    x_1& x_2& x_3\\
    T_1&T_2&T_3
\end{bmatrix}$$
and this is also easily verified by hand or a computer algebra system, such as \textit{Macaulay2} \cite{Macaulay2}.
\end{enumerate}
\end{ex}

As can be seen in the examples above, there is a symmetry between $f=f_0$ and $f_d$. Indeed, the latter is precisely the former's counterpart in $k[T_1,\ldots,T_n]$. Following the remark above, it is clear that this happens from the viewpoint of an exchange process.

\begin{rem}\label{Jd contained in J}
  For $0\leq i\leq d$, we have $\J_i \subseteq \J$.
\end{rem}

\begin{proof}
We proceed by induction once more. The claim is clear when $i=0$ by \Cref{J correct colon prop}, so suppose $1\leq i\leq d$ and that $\J_{i-1}=I_2(\psi)+(f_0,\ldots,f_{i-1})\subseteq \J$. We claim that $f_i \in \J_{i-1}:\n$, which we show by creating an expression similar to (\ref{f1 in colon}). Indeed, from the definition of $f_i$ and $\partial f_{i-1}$, for any $x_j$ we have
\begin{equation}\label{fi in colon}
x_jf_i = [x_jT_1\ldots x_j T_n] \cdot \partial f_{i-1} = [x_1T_j\ldots x_n T_j] \cdot \partial f_{i-1} +[(x_jT_1- x_1T_j)\ldots (x_jT_n- x_n T_j)] \cdot \partial f_{i-1}.
\end{equation}
The first summand is exactly $T_jf_{i-1}\in \J_{i-1}$ and the second summand is contained in $I_2(\psi) \subseteq \J_{i-1}$, hence $f_i \in \J_{i-1}:\n$. Since $\J_i = \J_{i-1}+(f_i)$, we have that $\J_i \subseteq \J_{i-1}:\n \subseteq \J:\n = \J$,  where the second containment follows from the induction hypothesis and the equality holds since $\J$ is saturated by \Cref{J a saturation}.
\end{proof}

Now that it has been shown that $\J_d \subseteq \J$, recall from \Cref{J=D} that $\widetilde{\J} =\frac{\widetilde{f}\widetilde{\K}^{(d)} }{\widetilde{x_n}^d}$. We show that there is a similar description for $\widetilde{\J_d}$.

\begin{thm}\label{Jd divisorial}
With the assumptions of \Cref{main setting} and $\J_d$ the ideal of the full downgraded sequence of \Cref{algorithm defn}, we have $\widetilde{\J_d} = \frac{\widetilde{f}\widetilde{\K}^{d} }{\widetilde{x_n}^d}$.
\end{thm}

\begin{proof}
Recall that $\widetilde{\K} = (\widetilde{x_n},\widetilde{T_n})$, hence we must show that $\widetilde{f_0} (\widetilde{x_n},\widetilde{T_n})^d =\widetilde{x_n}^d(\widetilde{f_0},\ldots,\widetilde{f_d})$. In particular, it suffices to show that $\widetilde{f_0} \widetilde{x_n}^{d-i}\widetilde{T_n}^i = \widetilde{x_n}^d\widetilde{f_i}$  or rather, since $\R(\n)$ is a domain, that $\widetilde{f_0} \widetilde{T_n}^i = \widetilde{x_n}^i\widetilde{f_i}$ for $0\leq i\leq d$.
There is nothing to be shown for the case $i=0$, hence we take $i=1$ to be the initial case, and proceed by induction. Recall that $f_0 = [x_1\ldots x_n]\cdot \partial f_0$ and $f_1 = [T_1\ldots T_n]\cdot \partial f_0$. Thus modulo $I_2(\psi)$, we have 
    $$\widetilde{f_0} \widetilde{T_n}  = [\widetilde{T_n}\widetilde{x_1}\ldots \widetilde{T_n}\widetilde{x_n}]\cdot \widetilde{\partial f_0} = [\widetilde{T_1}\widetilde{x_n}\ldots \widetilde{T_n}\widetilde{x_n}]\cdot \widetilde{\partial f_0}= \widetilde{x_n} \widetilde{f_1}$$
    and the initial claim follows.

If $d=1$ we are finished, so suppose that $2\leq i\leq d$ and the claim holds up to $i-1$. In order to show that $\widetilde{f_0}\widetilde{T_n}^i = \widetilde{x_n}^i\widetilde{f_i}$, we first note that, from the induction hypothesis, we have 
$$\widetilde{f_0} \widetilde{T_n}^i = \widetilde{f_0}\widetilde{T_n}^{i-1} \cdot \widetilde{T_n} = \widetilde{x_n}^{i-1}\widetilde{f_{i-1}}\cdot \widetilde{T_n} = \widetilde{x_n}^{i-1}\cdot\widetilde{f_{i-1}} \widetilde{T_n}. $$
Hence we will be finished once it has been shown that $\widetilde{f_{i-1}}\widetilde{T_n} = \widetilde{x_n}\widetilde{f_i}$, which follows exactly as before. Indeed, since
$f_{i-1} = [x_1\ldots x_n]\cdot \partial f_{i-1}$ and $f_i = [T_1\ldots T_n]\cdot \partial f_{i-1}$, modulo $I_2(\psi)$ we have 
$$\widetilde{f_{i-1}}\widetilde{T_n} = [\widetilde{T_n}\widetilde{x_1}\ldots \widetilde{T_n}\widetilde{x_n}]\cdot \widetilde{\partial f}_{i-1} = [\widetilde{T_1}\widetilde{x_n}\ldots \widetilde{T_n}\widetilde{x_n}]\cdot \widetilde{\partial f}_{i-1}= \widetilde{x_n}\widetilde{f_i}$$
which completes the proof.
\end{proof}

With the description of $\widetilde{\J_d}$ above, we now give a criterion for when the algorithm of downgraded sequences yields a generating set of the defining ideal $\J$.

\begin{cor}\label{Equal criteria}
With the assumptions of \Cref{main setting}, $\J= \J_d$ in $S[T_1,\ldots,T_n]$ if and only if $\widetilde{\K}^{d} = \widetilde{\K}^{(d)}$ in $\R(\n)$.
\end{cor}

\begin{proof}
This follows from \Cref{Jd contained in J}, \Cref{J=D}, and \Cref{Jd divisorial}.
\end{proof}

%%%%%%%%%%%%%%%%%%%%%%%%%%%%%%%%%%%%%%%%%%%%%%%%%%%%%%%%%%%%%%%%%%%%%%%%%%%%%%%%%%%%%%%%%%%%%%%%%%%%%%%%%%%%%%%%%%%%%%%%%%%%%%%%%%%%%%%%%%%%%%%%%%%%%%%%%%%%%%%%%%%%%%%%%%%%%%%%%%%%%%%%%%%%%%%%%%%%%%%%%%%%%%%%%%%%%%%%%%%%%%%%%%%%%%%%%%%%%%%%%%%%%%%%%%%%%%%%%%%%%%%%%%%%%%%%%%%%%%%%%%%%%%%%%%%%%%%%%%%%%%%%%%%%%%%%%%%%%%%%%%%%%%%%%%%%%%%%%%%%%%%%%%%%%%%%%%%%%%%%%%%%%%%%%%%%%%%%%%%

\section{Defining Ideal}\label{Defining Ideal Section}

In \Cref{Equal criteria}, a condition was given for when the defining ideal $\J$ coincides with the ideal $\J_d$ of the full downgraded sequence from \Cref{algorithm defn}. We begin this section by showing that this condition is satisfied. We then devote the remainder of the section to exploring various properties of $\R(\m)$, such as relation type, Cohen--Macaulayness, and Castelnuovo--Mumford regularity.

\begin{thm}\label{main result}
With the assumptions of \Cref{main setting}, the defining ideal $\J$ of $\R(\m)$, as a quotient of $S[T_1,\ldots,T_n]$, is $\J = I_2(\psi) +(f_0,\ldots,f_d)$.
\end{thm}

\begin{proof}
As noted in \Cref{Equal criteria}, it suffices to show that $\widetilde{\K}^{d} = \widetilde{\K}^{(d)}$. Recall from \Cref{K SCM and generically a CI} that $\widetilde{\K}=(\widetilde{x_n},\widetilde{T_n})$ is a strongly Cohen--Macaulay ideal of height one, and is also generically a complete intersection. Thus by \cite[3.4]{SV81}, it suffices to show that $\widetilde{\K}_\p$ is a principal ideal, for any prime $\R(\n)$-ideal $\p$ with $\hgt \p =2$.

First note that if $\p \nsupseteq \widetilde{\n}$, the proof of \Cref{J=D} shows that $\widetilde{\K}_\p = (\widetilde{x_n})_\p$. Hence we may assume that $\p \supseteq \widetilde{\n}$. In this case, there must exist a $\widetilde{T_i}$ such that $\widetilde{T_i} \notin \p$. Indeed, otherwise there is a containment $(x_1,\ldots, x_n,T_1,\ldots,T_n)^{\widetilde{\,\,}} \subseteq \p$. However, this is impossible as the former ideal has height $2n-(n-1) = n+1 \geq 3$ (recall that $n\geq 2$) and $\hgt \p =2$. For any such $\widetilde{T_i}\notin \p$, we note that $\widetilde{T_i}$ is then a unit locally at $\p$. Thus from the equality $\widetilde{x_i}\widetilde{T_n} = \widetilde{x_n}\widetilde{T_i}$ modulo $I_2(\psi)$, we have $\widetilde{\K}_\p = (\widetilde{x_n},\widetilde{T_n})_\p = (\widetilde{T_n})_\p$, and the claim follows.
\end{proof}

\begin{prop}\label{min gen set}
If $d\geq 2$, the ideal $\J$ is minimally generated as $\J=I_2(\psi)+(f_0,\ldots,f_d)$. The defining ideal of $\R(\m)$, as a quotient of $R[T_1,\ldots,T_n]$ is then minimally generated as $\overline{\J}=\overline{I_2(\psi)+(f_1,\ldots,f_d)}$.
\end{prop}

\begin{proof}
It suffices to prove the first statement, as the second follows immediately as $f$ is then part of a minimal generating set of $\J$. Writing $\ell_1,\ldots,\ell_{\scalebox{0.6}{$\binom{n}{2}$}}$ to denote the minors of $\psi$, we note that these polynomials have bidegree $(1,1)$. Moreover, recall from \Cref{fi bideg remark} that $\bideg f_i = (d-i,i)$ for $0\leq i\leq d$. With this, we show that $\{\ell_1,\ldots,\ell_{\scalebox{0.6}{$\binom{n}{2}$}}, f_0,\ldots,f_d\}$ is a minimal generating set of $\J$, from bidegree considerations.

We first show that it is impossible for any $\ell_i$ to be a non-minimal generator. Without loss of generality, suppose that $\ell_1$ is non-minimal and can hence be written in terms of the remaining generators. If $d\geq 3$, then from bidegree considerations and noting that either $d-i>1$ or $i> 1$, it follows that $\ell_1 \in (\ell_2,\ldots,\ell_{\scalebox{0.6}{$\binom{n}{2}$}})$, which is impossible by \cite[A2.12]{Eisenbud}. If $d=2$, then noting that $\bideg f_1 = (1,1)$ in this case and considering all other bidegrees, it follows that $\ell_1$ may be written in terms of $\{\ell_2,\ldots,\ell_{\scalebox{0.6}{$\binom{n}{2}$}}, f_1\}$. 
Since each polynomial here has bidegree $(1,1)$, such an expression could be rearranged to show that $f_1\in I_2(\psi)$. However, this is impossible as then $\partial f_1$ is a syzygy on $T_1,\ldots,T_n$ and so $f_2=f_d =0$ which is a contradiction to \Cref{fi bideg remark}.

Now that $\ell_1,\ldots,\ell_{\scalebox{0.6}{$\binom{n}{2}$}}$ have been shown to be minimal generators of $\J$, we show the same for $f_0,\ldots,f_d$. Note that $\bideg f_0 = (d,0)$ and $\bideg f_d=(0,d)$, hence from bidegree considerations it is clear that these polynomials are minimal generators of $\J$. Now suppose that some $f_j$, for $1\leq j\leq d-1$, is non-minimal and can be written in terms of the remaining generators. However, with \Cref{fi bideg remark} and considering both components in the bigrading it follows that $f_j\in I_2(\psi)$. As before, this is impossible as $j\leq d-1$ and the subsequent equations would then vanish.
\end{proof}

\begin{rem}\label{d=1 case}
Although the case $d=1$ is excluded in \Cref{min gen set}, this setting is of little concern. Indeed, after a linear change of coordinates, we may take $f=x_n$, and so $R\cong k[x_1,\ldots,x_{n-1}]$ with $\m=(x_1,\ldots,x_{n-1})$. Moreover, we then have $f_1 =T_n$, and so $\J = \K = \L$. With this and \Cref{main result}, we see that
$$\R(\m) \cong \frac{S[T_1,\ldots,T_n]}{I_2(\psi) + (x_n,T_n)} \cong \frac{R[T_1,\ldots,T_{n-1}]}{I_2(\psi')}$$
where 
$$\psi' =\begin{bmatrix}
   x_1&\ldots &x_{n-1}\\
    T_1&\ldots& T_{n-1}
\end{bmatrix}$$
as in the description of $\K$ in (\ref{K and matrix}). In particular, the description of $\R(\m)$ in \Cref{main result} recovers the original result of Micali \cite[Chap. I: Thm. 1, Lem. 2]{Micali64} in this case.   
\end{rem}

Recall that the \textit{relation type} $\rt(\m)$ is the largest degree, with respect to $T_1,\ldots,T_n$, within a minimal generating set of the defining ideal of $\R(\m)$.

\begin{cor}\label{rel type cor}
With the assumptions of \Cref{main setting}, the relation type of $\m$ is $\rt(\m) =d$.
\end{cor}

\begin{proof}
    If $d\geq 2$, this follows from \Cref{min gen set}. If $d=1$ then, as noted in \Cref{d=1 case}, $\m$ is of linear type and so $\rt(\m)=1$.
\end{proof}

\begin{cor}\label{fiber ring cor}
With the assumptions of \Cref{main setting}, the fiber ring of $\m$ is $\F(\m) \cong k[T_1,\ldots,T_n]/(f_d)$. Moreover, $\F(\m)$ is Cohen--Macaulay.
 \end{cor}

\begin{proof}
From \Cref{main result}, we have $\J = I_2(\psi) +(f_0,\ldots,f_d)$. As noted in the proof of \Cref{min gen set}, $I_2(\psi)$ is generated in bidegree $(1,1)$ and each $f_i$ is nonzero of bidegree $(d-i,i)$ by \Cref{fi bideg remark}. Hence from bidegree considerations, we have $\J + \n =(f_d) +\n$, and so $\F(\m) \cong S[T_1,\ldots,T_n]/(\J+\n) \cong k[T_1,\ldots,T_n]/(f_d)$. The Cohen--Macaulayness of $\F(\m)$ is clear, as it is a hypersurface ring.
\end{proof}

\subsection{Cohen--Macaulayness}

We now discuss the Cohen--Macaulayness of $\R(\m)$, and begin by constructing a handful of short exact sequences, similar to the procedure in \cite{BM16}. As depth is easily compared along short exact sequences, we are able to bound the depths of all modules involved.

With $\n =(x_1,\ldots,x_n)$ as before, recall that $\widetilde{\K} = (\widetilde{x_n},\widetilde{T_n})$ and this is a strongly Cohen--Macaulay ideal with $\widetilde{\K} = (\widetilde{x_n}):\widetilde{\n}$ by \Cref{K SCM and generically a CI} and \Cref{linkage prop}.
As such, there is a short exact sequence of bigraded $\R(\n)$-modules
\[
0\longrightarrow \n \R(\n)(0,-1) \longrightarrow \R(\n)(-1,0)\oplus \R(\n)(0,-1)\longrightarrow \widetilde{\K} \longrightarrow 0.
\]
Now applying the symmetric algebra functor $\S(-)$ to this sequence and taking the $d$th graded strand, we obtain
\[
\n\R(\n) (0,-1) \otimes \S_{d-1}\big(\R(\n)(-1,0)\oplus \R(\n)(0,-1)\big) \rightarrow \S_d\big(\R(\n)(-1,0)\oplus \R(\n)(0,-1)\big)\rightarrow \S_d(\widetilde{\K}) \rightarrow 0.
\]

As $\widetilde{\K}=(\widetilde{x_n},\widetilde{T_n})$ is a strongly Cohen--Macaulay $\R(\n)$-ideal and is 2-generated, it follows that $\widetilde{\K}$ is of linear type \cite[2.6]{HSV82}, hence $\S_d(\widetilde{\K}) = \widetilde{\K}^d$. Additionally, by rank considerations, the kernel of the first map above is torsion. As it is a submodule of a torsion-free module, this kernel must vanish, and so we have a short exact sequence
\[
0\rightarrow\n\R(\n) (0,-1) \otimes \S_{d-1}\big(\R(\n)(-1,0)\oplus \R(\n)(0,-1)\big) \rightarrow \S_d\big(\R(\n)(-1,0)\oplus \R(\n)(0,-1)\big)\rightarrow \widetilde{\K}^d \rightarrow 0.
\]
Now writing the symmetric powers above as free $\R(\n)$-modules, we then have
\begin{equation}\label{direct sum SES}
0\longrightarrow \bigoplus_{i=0}^{d-1} \n \R(\n)\big(-i,-(d-i)\big)\longrightarrow \bigoplus_{i=0}^{d} \R(\n)\big(-i,-(d-i)\big)
\longrightarrow \widetilde{\K}^d \longrightarrow 0.
\end{equation}

Additionally, we have the following natural short exact sequences
\begin{equation}\label{fiber ring SES}
0\longrightarrow\n \R(\n) \longrightarrow \R(\n) \longrightarrow \F(\n) \cong k[T_1,\ldots,T_n]\longrightarrow 0
\end{equation}
and 
\begin{equation}\label{Jtilde SES}
0\longrightarrow\widetilde{\J} \longrightarrow \R(\n) \longrightarrow \R(\m) \longrightarrow 0.
\end{equation}

With the short exact sequences (\ref{direct sum SES}) -- (\ref{Jtilde SES}), we may apply \cite[18.6]{Eisenbud} multiple times to compare the depths of the modules and rings above. Recall that a Noetherian local ring $A$ is said to be \textit{almost} Cohen--Macaulay if $\dim A -\depth A \leq 1$. In particular, we note that Cohen--Macaulay implies almost Cohen--Macaulay.

\begin{thm}\label{Cohen--Macaulayness of R(m)}
With the assumptions of \Cref{main setting}, $\R(\m)$ is almost Cohen--Macaulay, and is Cohen--Macaulay if and only if $d\leq n-1$. 
\end{thm}

\begin{proof}
We begin by showing that $\R(\m)$ is almost Cohen--Macaulay, i.e. that $\depth \R(\m) \geq n-1$. Recall that $\R(\n)$ is a Cohen--Macaulay domain with $\dim \R(\n) = n+1$. Thus from (\ref{fiber ring SES}), it follows that $\depth \n\R(\n) \geq n+1$, hence $\depth \n\R(\n) = n+1$ as this is the maximum. Now with this, and comparing depths in (\ref{direct sum SES}), it then follows that $\depth \widetilde{\K}^d \geq n$. With this and the bigraded isomorphism $\widetilde{\J} =\frac{\widetilde{f}\widetilde{\K}^{d} }{\widetilde{x_n}^d} \cong\widetilde{\K}^{d}$, comparing depths along (\ref{Jtilde SES}) shows that $\depth \R(\m) \geq n-1$, as claimed. 

For the second assertion, if $\R(\m)$ is Cohen--Macaulay, then by \cite[6.3]{AHT95} we have $\rt(\m) \leq \dim \F(\m)$. Hence by \Cref{rel type cor} and \Cref{fiber ring cor} we have $d\leq n-1$. For the converse, we again note that there is an isomorphism $\widetilde{\J}\cong \widetilde{\K}^d$. Moreover, from the definition of $\K$, it follows from \cite[A2.13, A2.14]{Eisenbud} that $\widetilde{\K}^i$ is a maximal Cohen--Macaulay $\R(\n)$-module, and hence a Cohen--Macaulay $\R(\n)$-ideal, for $1\leq i\leq n-1$. Thus if $d\leq n-1$, then $\R(\n) /\widetilde{\J} \cong \R(\m)$ is Cohen--Macaulay.
\end{proof}

\subsection{Regularity}
With the bigraded short exact sequences (\ref{direct sum SES}) -- (\ref{Jtilde SES}), we may also bound the \textit{Castelnuovo-Mumford regularity} of the Rees algebra $\R(\m)$, in a manner similar as was done for depth. We refer the reader to \cite{Trung98} for all definitions and conventions regarding regularity. Additionally, we note that regularity can be determined from a minimal graded free resolution, and is also easily compared along short exact sequences \cite[\textsection 20.5]{Eisenbud}.

We first note that the modules involved are bigraded, and so there are multiple choices for a single grading. We write $\reg_{x}$ to denote regularity with respect to the grading given by $\deg x_i =1$ and $\deg T_i =0$. Similarly, write $\reg_{T}$ to denote regularity with respect to the grading given by $\deg x_i =0$ and $\deg T_i =1$. Lastly, write $\reg_{x,T}$ to denote regularity with respect to the total grading given by $\deg x_i = \deg T_i =1$.

\begin{thm}\label{regularity thm}
With the assumptions of \Cref{main setting}, we have
\[
\begin{array}{ccc}
    \reg_{T}  \R(\m)  = d-1, &  \reg_{x}  \R(\m)  \leq  d-1, &  \reg_{x,T}  \R(\m)\leq d.
\end{array}
\]
Additionally, $\reg_{T} \F(\m) = d-1$. 
\end{thm}

\begin{proof}
We first note that the claim on the regularity of $\F(\m)$ follows as $\F(\m) \cong k[T_1,\ldots,T_n]/(f_d)$ and $\deg f_d =d$, by \Cref{fiber ring cor}. Moreover, it is well known that $\rt(\m) -1 \leq \reg_{T}  \R(\m)$ \cite{Trung98}, hence by \Cref{rel type cor} we have $d-1 \leq \reg_{T}  \R(\m)$. Thus the first equality will follow once it has been shown that $\reg_{T}  \R(\m) \leq d-1$. We show this and the remaining inequalities simultaneously, by applying \cite[20.19]{Eisenbud} repeatedly.

Recall that $\R(\n) \cong S[T_1,\ldots,T_n]/I_2(\psi)$, for $\psi$ as in (\ref{psi defn}), and so $\R(\n)$ is resolved by the Eagon--Northcott complex. Writing this as a graded free resolution, with respect to the various gradings, we deduce that
\begin{equation}\label{reg R(n)}
\begin{array}{ccc}
    \reg_{T}  \R(\n)  = 0, &  \reg_{x}  \R(\n) =0, &  \reg_{x,T}  \R(\n)=1.
\end{array}
\end{equation}
Moreover, recall that $\F(\n) \cong k[T_1,\ldots,T_n]$. Thus we have  
\begin{equation}\label{reg F(n)}
\begin{array}{ccc}
    \reg_{T}  \F(\n)  = 0, &  \reg_{x}  \F(\n) =0, &  \reg_{x,T}  \F(\n)=0,
\end{array}
\end{equation}
as well. Now with these bounds, and the sequence (\ref{fiber ring SES}), it follows that 
\begin{equation}\label{reg nR(n)}
\begin{array}{ccc}
    \reg_{T}  \n\R(\n)  \leq 1, &  \reg_{x}  \n\R(\n) \leq 1, &  \reg_{x,T}  \n\R(\n)\leq 1.
\end{array}
\end{equation}

Now with (\ref{reg R(n)}) and (\ref{reg nR(n)}), we are able to bound the regularity of the modules in the sequence (\ref{direct sum SES}). For ease of notation, write
\[
\begin{array}{ccc}
    M= \displaystyle{\bigoplus_{i=0}^{d-1} \n \R(\n)\big(-i,-(d-i)\big)},  & \quad&  N= \displaystyle{\bigoplus_{i=0}^{d} \R(\n)\big(-i,-(d-i)\big)}  
\end{array}
\]
for the modules appearing in (\ref{direct sum SES}). With the inequalities above, it follows that
\begin{equation}\label{reg M and N}
\begin{array}{lcl}
   \reg_{T} M \leq d+1,  &\quad\quad &\reg_{T} N =d, \\[1ex]
   \reg_{x} M \leq d,  &\quad\quad &\reg_{x} N = d, \\[1ex]
   \reg_{x,T} M \leq d+1,  &\quad\quad  &\reg_{x,T} N =d+1.
     \end{array}
\end{equation}
With the bounds above and the sequence (\ref{direct sum SES}), it then follows that 
\begin{equation}\label{reg Kd}
\begin{array}{ccc}
    \reg_{T}  \widetilde{\K}^d  \leq d, &  \reg_{x}  \widetilde{\K}^d \leq d, &  \reg_{x,T}  \widetilde{\K}^d\leq d+1.
\end{array}
\end{equation}

Lastly, with the bigraded isomorphism $\widetilde{\J} =\frac{\widetilde{f}\widetilde{\K}^{d} }{\widetilde{x_n}^d} \cong\widetilde{\K}^{d}$ and the inequalities in (\ref{reg R(n)}) and (\ref{reg Kd}), comparing regularity in the sequence (\ref{Jtilde SES}) shows that 
\[
\begin{array}{ccc}
    \reg_{T}  \R(\m)  \leq d-1, &  \reg_{x}  \R(\m)  \leq  d-1, &  \reg_{x,T}  \R(\m)\leq d,
\end{array}
\]
which completes the proof.
\end{proof}

%%%%%%%%%%%%%%%%%%%%%%%%%%%%%%%%%%%%%%%%%%%%%%%%%%%%%%%%%%%%%%%%%%%%%%%%%%%%%%%%%%%%%%%%%%%%%%%%%%%%%%%%%%%%%%%%%%%%%%%%%%%%%%%%%%%%%%%%%%%%%%%%%%%%%%%%%%%%%%%%%%%%%%%%%%%%%%%%%%%%%%%%%%%%%%%%%%%%%%%%%%%%%%%%%%%%%%%%%%%%%%%%%%%%%%%%%%%%%%%%%%%%%%%%%%%%%%%%%%%%%%%%%%%%%%%%%%%%%%%%%%%%%%%%%%%%%%%%%%%%%%%%%%%%%%%%%%%%%%%%%%%%%%%%%%%%%%%%%%%%%%%%%%%%%%%%%%%%%%%%%%%%%%%%%%%%%%%%%%%

\end{document}